\documentclass{elsarticle}
\usepackage{anysize}
\marginsize{1in}{1in}{1in}{1in}
\usepackage{amssymb}
\usepackage{amsmath}
\usepackage{amsthm}
\usepackage{mathtools}
\usepackage{enumitem}
\newtheorem{defn}{Definition}
\newtheorem{thm}{Theorem}
\newtheorem{lemma}{Lemma}
\newtheorem{prop}{Proposition}
\newtheorem{note}{Note}

\newtheorem{nota}{Notation}
\newtheorem{case}{Case}

\newtheorem{remark}{Remark}
\theoremstyle{definition}
\begin{document}
\begin{abstract}
We deal with various Diophantine equations involving the Euler totient function and various sequences of numbers, including factorials, powers, and Fibonacci sequences.
\end{abstract}
\begin{frontmatter}
\title{Diophantine Equations Involving the Euler Totient Function\tnoteref{t1}}
\author[add1]{J.C. Saunders}
\ead{saunders@post.bgu.ac.il}
\tnotetext[t1]{Research of J.C. Saunders was supported by an Azrieli International Postdoctoral Fellowship}
\address[add1]{Department of Mathematics, Ben Gurion University of the Negev, Be'er Sheva, Israel 8410501}
\end{frontmatter}
\section{Introduction}
There has been much study on diophantine equations involving factorials and powers. Some examples of such equations are:
\newline
1) $n!=x^k\pm y^k$ and $n!\pm m!=x^k$ \cite{erdos}.
\newline
2) $dx^k=u_n$, where $(u_n)_{n=1}^{\infty}$ satisfies a linear recurrence of order $2$ with constant coefficients \cite{shorey}.
\newline
3) $x^k=F_n\pm 1$, where $(F_n)_{n=1}^{\infty}$ is the Fibonacci sequence \cite{bugeaud}.
\newline
4) $F_n^{(k)}=2^m$, where $(F_n^{(k)})_n$ is the $k$-bonacci sequence, given by $F_n^{(k)}=0$ for $-(k-2)\leq n\leq 0$, $F_1^{(k)}=1$, and $F_n^{(k)}=F_{n-1}^{(k)}+F_{n-2}^{(k)}+\ldots+F_{n-k}^{(k)}$ for $n>1$ \cite{bravo}.
\newline
5) $\varphi(|u_n|)=2^m$ \cite{damir}.
\newline
6) $\varphi(x)=n!$ \cite{pomerance}.
\newline
7) $P(x)=n!$, where $P\in\mathbb{Z}[x]$ \cite{berend}.
\newline
8) $\varphi(F_n)=m!$ \cite{luca}.
\newline
9) $\varphi(L_n)=2^x3^y$, where $L_n$ is the Lucas sequence \cite{luca}.

For several of the equations above, it was proved that the number of solutions is at most finite, and in some cases that all solutions can be effectively found. Others, for example the equation $\varphi(x)=n!$, do have infinitely many solutions.

Our starting point in this paper is Luca and St\u anic\u a's results \cite{luca}. It is natural to ask for which polynomials $P(x)$, the diophantine equation $\varphi(P(x))=n!$ has only finitely many solutions. Here we answer this question when $P(x)$ is a monomial, i.e., $P(x)=ax^m$ for some $m\geq 2$, $a\in\mathbb{N}$, and explicitly give all of the solutions to this equation. We also generalise Luca and St\u anic\u a's results to certain Lucas sequences of the first and second kinds, defined as follows.
\begin{defn}
Let $a,b,c\in\mathbb{N}$. A \emph{Lucas sequence of the first kind} $(u_n)_n$ is defined by $u_0=0$, $u_1=1$, and $u_n=bu_{n-1}+cu_{n-2}$ for all $n\geq 2$. Likewise, define the sequence $(v_n)_n$ by $v_0=2$, $v_1=b$, and $v_n=bv_{n-1}+cv_{n-2}$, which is a \emph{Lucas sequence of the second kind}.
\end{defn}
We prove that the Euler function, evaluated at the $p$th term of Lucas sequences, where $p$ is prime, is a factorial only finitely often, and give bounds on such primes $p$. Also, for three specific sequences, we give all of the solutions as to when the Euler function evaluated at the terms is of the form $2^x3^y$.

In Section $2$, we state these results, and in Section $3$ give their proofs.
\section{The Main Results}
Our first two results involve factorials, monomials, and Euler's totient function.
\begin{thm}\label{thm:phiresult1}
Fix $a,b,c,m\in\mathbb{N}$ with $\gcd(b,c)=1$ and $m\geq 2$. Then there are only finitely many solutions to $\varphi(ax^m)=\frac{b\cdot n!}{c}$, and these solutions satisfy $n\leq\max\{61,3a,3b,3c\}$. In particular, all of the integer solutions to $\varphi(x^m)=n!$, where $m\geq 2$, are $\varphi(1^m)=1!$, $\varphi(2^2)=2!$, $\varphi(3^2)=3!$, $\varphi((3\cdot 5)^2)=5!$, $\varphi((3\cdot 5\cdot 7)^2)=7!$, $\varphi((2^2\cdot 3\cdot 5\cdot 7)^2)=8!$, $\varphi((2^2\cdot 3^2\cdot 5\cdot 7)^2)=9!$, $\varphi((2^2\cdot 3^2\cdot 5\cdot 7\cdot 11)^2)=11!$, and $\varphi((2^2\cdot 3^2\cdot 5\cdot 7\cdot 11\cdot 13)^2)=13!$.
\end{thm}
In the ``other direction", we can prove the following.
\begin{thm}\label{thm:phiresult2}
Fix $a,b,c,m\in\mathbb{N}$ with $\gcd(a,b)=1$ and $m\geq 2$. Then there are only finitely many solutions to $\varphi\left(\frac{a\cdot n!}{b}\right)=cx^m$, and these solutions satisfy $n\leq\max\{61,3a,3b,3c\}$. In particular, all of the integer solutions to $\varphi(n!)=x^m$, where $m\geq 2$ and $n\geq 1$, are $\varphi(1!)=1^m$, $\varphi(2!)=1^m$, $\varphi(4!)=2^3$, $\varphi(5!)=2^5$, $\varphi(8!)=(2^5\cdot 3)^2$, $\varphi(9!)=(2^5\cdot 3^2)^2$, $\varphi(11!)=(2^6\cdot 3^2\cdot 5)^2$, and $\varphi(13!)=(2^8\cdot 3^3\cdot 5)^2$.
\end{thm}
\begin{remark}
While Theorem $1$ states that the equation $\varphi(x^m)=n!$ has only finitely many integer solutions for $m\geq 2$, Erd\H{o}s \cite[p.~144]{guy} observed that the equation does have many solutions when $m=1$. Indeed, Ford, et al. \cite{pomerance} observed that there exists $c>0$ such that, for all sufficiently large $k\in\mathbb{N}$, the number of solutions to $\varphi(x)=k!$ is at least $(k!)^c$.
\end{remark}
The remaining results generalise Luca's and St\u anic\u a's results \cite{luca}.
\begin{thm}\label{thm:phiresult4}
Let $a,b,c\in\mathbb{N}$ with $b^2+4c$ being prime and $b^2+4c>a$. Let $(u_n)_n$ be a Lucas sequence of the first kind. Then there are at most finitely many primes $p$ for which $\varphi(au_p)$ is a factorial. Moreover, such primes $p$ are bounded above by
\begin{equation*}
\max\left\{ea^{1/2}\left(\frac{b+\sqrt{b^2+4c}}{2}\right),\frac{\frac{10}{9}\log(8\cdot(b^2+4c-1)!)-\log a+\frac{\log(b^2+4c)}{2}}{\log\left(\frac{b+\sqrt{b^2+4c}}{2}\right)}\right\}.
\end{equation*}
\end{thm}
The bounds in Theorem \ref{thm:phiresult4} approach $\infty$ as $a,b$, and/or $c$ approach $\infty$, but only grow polynomially fast in terms of $a,b$, and $c$.

For any specific values of $b$ and $c$, finding all of the solutions to the equation $\varphi(v_n)=2^x3^y$, where $(v_n)_n$ is a Lucas sequence of the second kind, is non-trivial, since there are potentially infinitely many solutions. For three pairs of specific values of $b$ and $c$, however, we prove that this equation only has finitely many solutions and explicitly list all of them.
\begin{thm}\label{thm:phiresult5}
Let $(v_n)_n$ be a Lucas sequence of the second kind. The only solutions to $\varphi(v_n)=2^x3^y$ are:
1) For $b=3$, $c=1$:
\begin{equation*}
(n,x,y)=(0,0,0),(1,1,0),(3,2,1),(4,5,1),(9,6,5).
\end{equation*}
2) For $b=5$, $c=1$:
\begin{equation*}
(n,x,y)=(0,0,0),(1,2,0),(2,0,2),(3,4,3).
\end{equation*}
3) For $b=7, c=1$:
\begin{equation*}
(n,x,y)=(0,0,0),(1,1,1),(2,5,0),(3,5,2),(6,9,4).
\end{equation*}
\end{thm}
\section{Proofs}
\begin{prop}\label{prop:1}
Let $x,n,m,a,b,c\in\mathbb{N}$ with $m\geq 2$ and $\varphi(ax^m)=\frac{b\cdot n!}{c}$ and let $p$ be a prime such that $p>a,b,c$. If $p\mid x$, then $p\leq n$. Conversely, if $p\leq n$, $p\nmid x$, then $p=2$ and $n=3,5$, or $7$.
\end{prop}
\begin{proof}
Suppose $p\mid x$. Then $p^2\mid ax^m$. Thus, $p\mid\varphi(ax^m)$, so that $p\mid\frac{b\cdot n!}{c}$. Thus, $p\mid n!$ so that $p\leq n$.

Conversely, suppose $p\leq n$ and $p\nmid x$. We have $p\mid\frac{b\cdot n!}{c}$ so that $p\mid\varphi(ax^m)$. Since $p\nmid x$, we have $p\nmid ax^m$. Thus, we must have that there exists a prime, say $p'$, such that $p\mid p'-1$ and $p'\mid ax^m$ so that $p'\mid x$. Let $q$ be the greatest prime at most $n$. Then $a,b,c<p\leq q$ so that $q\mid\frac{b\cdot n!}{c}$. Then $q\mid\varphi(ax^m)$. Either $q\mid ax^m$ or there exists a prime $q'\mid ax^m$ such that $q\mid q'-1$. Consider the latter case. Then we have $q'\mid x$ and $q'>q>a,b,c$. Thus, $q'^2\mid x^m$ so that $q'^2\mid ax^m$. Hence, $q'\mid\varphi(ax^m)=\frac{b\cdot n!}{c}$, and so $q'\mid n!$. But then $q<q'\leq n$, contradicting our choice of $q$. Thus, the former case must hold, and we have $q\mid x$. Using the same reasoning, we can deduce that the highest prime dividing $x$ is $q$. Observe that $a,b,c<p<p'\leq q\leq n$. We can therefore deduce that for all $e\in\mathbb{N}$ $p^e\mid\frac{c\cdot n!}{d}$ if and only if $p^e\mid(q_1-1)(q_2-1)\cdots(q_r-1)$ where $q_1<q_2<\ldots <q_r=q$ are all the primes dividing $x$ that are greater than $a,b$, and $c$. Thus, for all $e\in\mathbb{N}$ $p^e\mid n!$ if and only if $p^e\mid(q_1-1)(q_2-1)\cdots(q_r-1)$. Observe that $q_1-1<q_2-1<\ldots<q_r-1<n$ and that $p\nmid\frac{n!}{(q_1-1)\cdots(q_r-1)}$. Thus, 
 $q_1-1,\ldots,q_r-1$ must contain all of the positive multiples of $p$ up to $n$. We must therefore have that $p=q_i-1$ for some $1\leq i\leq r$, which can only hold if $p=2$. So $q_1-1,\ldots,q_r-1$ contains all of the positive even numbers less than $n$ and $n=q_r=p_k$. Thus, $n=3,5$, or $7$.
\end{proof}
For the next proposition, we require the following definition.
\begin{defn}
A number $n\in\mathbb{N}$ is a \emph{powerful number} if $n$ does not have a prime factor to the power $1$ in its prime factorisation.
\end{defn}
\begin{prop}\label{prop:2}
Let $x,y,\in\mathbb{N}$ satisfy $\varphi(x)=\varphi(y)$ and suppose that $x$ and $y$ are both powerful numbers. Then $x=y$.
\end{prop}
\begin{proof}
Let $P(n)$ denote the largest prime factor dividing $n$. For $x,y\in\mathbb{N}$ both powerful with $\varphi(x)=\varphi(y)$ implies that $P(x)=P(y)$ with their exponents in the factorisation of $x$ and $y$ being equal. The result then follows by induction on the number of prime factors of $x$.
\end{proof}
\begin{lemma}\label{lem:1}
If $x,n,a,b,c\in\mathbb{N}$ with $n\geq 9$, $a,b\leq n/3$, and $\varphi(ax^2)=\frac{b\cdot n!}{c}$, then all of the primes in the interval $(n/3,n/2]$ are congruent to $2\pmod 3$.
\end{lemma}
\begin{proof}
Let $p\in(n/3,n/2]$ be prime. By Proposition \ref{prop:1}, we have $p\mid x$. Thus, $p^{2e}\parallel ax^2$ for some $e\in\mathbb{N}$. Thus, $p^{2e-1}\mid\varphi(ax^2)$. Notice that $p^2\parallel\frac{b\cdot n!}{c}$. We can therefore deduce that there exists a prime $q\mid ax^2$ such that $p\mid q-1$. Notice that $q\mid x$, and so, by Proposition \ref{prop:1}, $q\leq n$. But since $n/3<p$ we must therefore have that $2p=q-1$. Since $n\geq 9$, we have $3\nmid p,q$. Thus, $p\equiv 2\pmod{3}$.
\end{proof}
We also have the following result of Rosser and Schoenfield \cite[p.~72]{rosser}.
\begin{lemma}[Rosser, Schoenfield]\label{lem:3}
Let $c$ be the Euler-Mascheroni constant
\begin{equation*}
c=\lim_{n\rightarrow\infty}\left(-\log n+\sum_{k=1}^n\frac{1}{k}\right)=0.57721\ldots
\end{equation*}
Then for all $n\geq 3$, we have
\begin{equation*}
n/\varphi(n)<e^c\log\log n+5/(2\log\log n)
\end{equation*}
except when $n=223092870=2\cdot 3\cdot 5\cdot 7\cdot 11\cdot 13\cdot 17\cdot 19\cdot 23$, in which case
\begin{equation*}
n/\varphi(n)<e^c\log\log n+2.50637/(\log\log n)
\end{equation*}
\end{lemma}
We use the following notation for the number of primes up to $x$ in a congruence class in the proofs of Theorems \ref{thm:phiresult1} and \ref{thm:phiresult2}:
\begin{nota}
For two coprime positive integers $a$ and $q$ and positive real number $x$, let $\pi(x;q,a)$ denote the number of primes up to $x$ that are congruent to $a\pmod q$. 
\end{nota}
\begin{proof}[Proof of Theorem \ref{thm:phiresult1}]
Suppose that $\varphi(x^m)=\frac{a}{b}\cdot n!$ where $m\geq 2$ and $\gcd(a,b)=1$. We divide into two cases.
\begin{case}$m\geq 3$

\normalfont
Suppose that $n>\max\{61,3a,3b,3c\}$. Let $p$ be the largest prime at most $n$. By Bertrand's Postulate, $n/2<p$ and so $p^2\nmid n!$. Also $p\nmid a,b$ since $a,b\leq\frac{n}{3}<p$. By Proposition \ref{prop:1}, we can see that $p\mid x$ and so $p^3\mid ax^m$. But then $p^2\mid\varphi(ax^m)=\frac{b\cdot n!}{c}$ so that $p^2\mid n!$, a contradiction.
\end{case}
\begin{case}$m=2$

\normalfont
Suppose that $n>\max\{61,3a,3b,3c\}$. Then, by Lemma \ref{lem:1}, all of the primes in the interval $(n/3,n/2]$ are congruent to $2\pmod 3$. Bennett, et al. \cite{bennett} showed that for $x\geq 450$, we have
\begin{equation*}
\frac{x}{2\log x}<\pi(x;3,1)<\frac{x}{2\log x}\left(1+\frac{5}{2\log x}\right).
\end{equation*}
Therefore, for $n\geq 1394$, we have
\begin{equation*}
\pi(n/2;3,1)-\pi(n/3;3,1)>\frac{n}{4\log(n/2)}-\frac{n}{6\log(n/3)}\left(1+\frac{5}{2\log(n/3)}\right)>0.
\end{equation*}
Thus, $n<1394$. Also, a quick check will confirm that for $62\leq n\leq 1393$ there eixsts a prime in the interval $(n/3,n/2]$ that is congruent to $1\pmod{3}$, contradicting all possibilities.
\end{case}
In both cases we have $n\leq\max\{61,3a,3b,3c\}$. Thus, by Lemma \ref{lem:3}, there are only finitely many solutions to $\varphi(ax^m)=\frac{b\cdot n!}{c}$. We find all of these solutions in the case $a=b=c=1$. We divide into several cases.
\addtocounter{case}{-2}
\begin{case}$m=3$, $x\geq 2$

\normalfont
Since $x\geq 2$, both $1$ and $x^m-1$ are coprime to $x$. Therefore, $n\geq 2$. Let $p$ be the largest prime at most $n$. By Bertrand's Postulate, $n/2<p$ and so $p^2\nmid n!$. By Proposition \ref{prop:1}, we can see that $p\mid x$ and so $p^3\mid x^m$. But then $p^2\mid\varphi(x^m)=n!$, a contradiction.
\end{case}

\begin{case}$m=2$, $n\geq 62$, $26\leq n\leq 56$, $14\leq n\leq 20$.

\normalfont
In these cases, $n\geq 9$ and so all of the primes in the interval $(n/3,n/2]$ are congruent to $2\pmod{3}$. Thus, $n\leq 61$. Also, a quick check will confirm that for $26\leq n\leq 56$, and $14\leq n\leq 20$ there exists a prime in the interval $(n/3,n/2]$ that is congruent to $1\pmod{3}$, contradicting all possibilities.
\end{case} 
\begin{case} $m=2$, $57\leq n\leq 61$.

\normalfont
By Proposition \ref{prop:1}, $11\mid x$. Suppose that $11^e\parallel x$. Then $11^{2e}\parallel x^2$. Also, $23\mid x$ and $23$ is the only prime up to $n$ that is congruent to $1\pmod{11}$. Thus, $11^{2e-1+1}\parallel\varphi(x^2)$ or $11^{2e}\parallel\varphi(x^2)$. But $11^5\parallel n!$, a contradiction since $5$ is odd.
\end{case}
\begin{case}$m=2$, $n=4,6,10,12,21,22,23,24,25$.

\normalfont
All of these cases are exhausted in the same way as the case of $57\leq n\leq 61$, but with a possibly different prime $p$ replacing $11$ for each one to derive that if $p^e\parallel\varphi(x^2)$ and $p^f\parallel n!$, then the parity of $e$ and $f$ differ, contradicting the specific case being considered. For the cases $n=4,21,23$ the prime $p$ is $2$, for cases $n=6,12,24,25$, the prime $p$ is $3$, for the case $n=10$, the prime $p$ is $5$, and for the case $n=22$, the prime $p$ is $11$.
\end{case}
\begin{case}$m=2$, $n=1,2,3,5,7,8,9,11,13$.

\normalfont
Proposition \ref{prop:2} gives the only solutions for these values of $n$ as stated in Theorem $1$.
\end{case}
\end{proof}
\begin{proof}[Proof of Theorem \ref{thm:phiresult2}]
Suppose that $\varphi\left(\frac{a\cdot n!}{b}\right)=cx^m$ where $m\geq 2$ and $\gcd(b,c)=1$. Suppose that $n>\max\{61,3a,3b,3c\}$. Bennett, et al. \cite{bennett} showed that for $x\geq 450$, we have
\begin{equation*}
\frac{x}{2\log x}<\pi(x;3,1)<\frac{x}{2\log x}\left(1+\frac{5}{2\log x}\right).
\end{equation*}
We can therefore derive that there exists a prime $p\in(n/3,n/2]$ that is congruent to $1\pmod{3}$. Then $p^2\parallel n!$, and so $p^2\parallel\frac{a\cdot n!}{b}$ since $p>n/3>a,b$. Thus, $p\mid cx^m$. Since $c<\frac{n}{3}<p$, we have $p\mid x^m$. But then $p^2\mid cx^m$. Therefore, there exists a prime $q\mid\frac{a\cdot n!}{b}$ such that $p\mid q-1$. Since $q>p>a$, we have that $q\mid n!$, and so $q\leq n$. Since $p\in(n/3,n/2]$, we therefore have that $2p=q-1$. But since $p\equiv 1\pmod 3$, we have $3\mid 2p+1$, a contradiction. In finding all of these solutions for $a=b=c=1$, it is therefore only necessary, by sole computation, to verify that for $n<62$ all of the solutions are as stated in the theorem.
\end{proof}
\begin{nota}
Lucas \cite{lucas} proved that for any prime $p$ not dividing $c$ we have that there exists $k\in\mathbb{N}$ such that $p\mid u_l$ if and only if $k\mid l$. Such a $k$ is called the index of appearance of $p$. Denote the index of appearance of a prime $p$ by $z(p)$.
\end{nota}
Lucas \cite{lucas} also proved the following.
\begin{lemma}[Lucas]\label{lem:4}
If $p\mid b^2+4c$, then $z(p)\mid p$. Let $p$ be a prime other than $b^2+4c$ with $p\nmid c$. If $b^2+4c$ is a quadratic residue $\pmod p$, then $z(p)\mid p-1$. If $b^2+4c$ is not a quadratic residue $\pmod p$, then $z(p)\mid p+1$. Let $\alpha=\frac{b+\sqrt{b^2+4c}}{2a}$ and $\beta=\frac{\sqrt{b^2+4c}-b}{2a}$. Then
\begin{equation*}
u_n=\frac{(\alpha^n-\beta^n)}{\sqrt{b^2+4c}}.
\end{equation*}
\end{lemma}
\begin{proof}[Proof of Theorem \ref{thm:phiresult4}]
Let $\varphi(au_p)=m!$. Suppose that $m\geq b^2+4c$. Then $b^2+4c\mid\varphi(au_p)$ so either $b^2+4c\mid au_p$ or there exists a prime $q\mid au_p$ such that $q\equiv 1\pmod{b^2+4c}$. In the former case, we thus have $p\mid b^2+4c$ and so $p=b^2+4c$. Thus, assume the latter case. Since $b^2+4c\equiv 1\pmod 4$ and $b^2+4c$ is prime, we have by quadratic reciprocity that $b^2+4c$ is a quadratic residue $\pmod{q}$. By Lemma \ref{lem:4}, we thus have that $z(q)\mid\gcd(p,q-1)$. We must have that $z(q)=p$ and so $p\mid q-1$. Thus, $p\mid m!$ so that $p\leq m$. By Lemma \ref{lem:4}, we have
\begin{equation*}
a\alpha^p>au_p>\varphi(au_p)\geq p!>(p/e)^p.
\end{equation*}
Since $p\geq 2$, we have $p<ea^{1/2}\alpha$.

Now assume that $m<b^2+4c$ and $p\geq ea^{1/2}\alpha$. Thus, $p\geq 5$. We can work out that $au_5=a(b^4+3b^2c+c^2)$ and so $u_p\geq u_5\geq 5$. Thus,
\begin{align*}
\frac{au_p}{(b^2+4c-1)!}\leq\frac{au_p}{m!}=\frac{au_p}{\varphi(au_p)}.
\end{align*}
The right-hand side of the above inequality can be bounded with Lemma \ref{lem:3} and the result can be deduced.
\end{proof}
\begin{note}
For the rest of the paper, let $a=c=1$.
\end{note}
We now list some properties of the sequences $(u_n)_n$ and $(v_n)_n$ that are well-known. We use the following notation for the highest power of a prime dividing $n\in\mathbb{N}$ in stating these properties and in the proof of Proposition \ref{prop:4}:
\begin{nota}
Let $v_p(n)$ denote the highest power of $p$ dividing $n$.
\end{nota}
For the 8 facts that follow $n$ is any natural number. Facts \ref{fact1} through \ref{fact5} can be found in \cite{ribenboim}. \ref{fact6} and \ref{fact7} follows routinely from \ref{fact1} through \ref{fact5}. \ref{fact8} follows routinely from the other facts together with Lemma $5$ from \cite{stewart}. 
\begin{enumerate}[label=\textbf{S.\arabic*}]
\item $v_n=\alpha^n+\beta^n$ where $\alpha=\frac{b+\sqrt{b^2+4}}{2}$ and $\beta=\frac{\sqrt{b^2+4}-b}{2}$\label{fact1}
\item $u_{2n}=u_nv_n$.\label{fact2}
\item $(b^2+4)u_n^2+4(-1)^n=v_n^2$. In particular, $\gcd(u_n,v_n)=1,2,4$.\label{fact3}
\item $v_{2n}=v_n^2-2(-1)^n$.\label{fact4}
\item Let $m$ be odd. Then $v_n\mid v_{nm}$.\label{fact5}
\item Let $\alpha\geq 1$ and $m\geq 1$ be odd. We have $v_{2^{\alpha}m}-v_{2^{\alpha}}=(b^2+4)u_{2^{\alpha-1}(m-1)}u_{2^{\alpha-1}(m+1)}$.\label{fact6}
\item $v_{3n}=v_n(v_n^2+1-2(-1)^n)$\label{fact7}
\item Let $d=\nu_3(b)$ if $3\mid b$ or $d=\nu_3(b^2+2)$ if $3\nmid b$. $3\mid u_n$ implies $\nu_3(u_n)=\nu_3(n)+d$.\label{fact8}
\end{enumerate}
In the proof of Proposition \ref{prop:4} we use the following notation for the Legendre symbol:
\begin{nota}
Let $(a|q)$ denote the Legendre symbol of $a$ with respect to the prime $q$.
\end{nota}
\begin{prop}\label{prop:4}
Let $c=1$ and $b^2+4$ be prime and let $d=\nu_3(b)$ if $3\mid b$ or $d=\nu_3(b^2+2)$ if $3\nmid b$. Suppose that $\varphi(v_n)=2^x3^y$ for some $x,y,n\geq 0$ and $n=2^em$ where $e\geq 0$ and $m$ is odd. Then $e\leq 2$ and at least one of the following conditions hold:

1) $n=0,1,2,3,4,6,12$

2) $n$ is a power of $3$

3) there exists a prime $p>3$ dividing $n$ and for all such primes $p$, there exist primes $q_1,\ldots,q_l$ such that $q_i=2\cdot 3^{b_{q_i}}+1$ for some $b_{q_i}\in\mathbb{N}$ for all $1\leq i\leq l$ with $v_{2^{e}p}=v_{2^{e}}q_1\cdots q_l$, but $q_i\nmid v_{2^{e}}$ for all $1\leq i\leq t$. Moreover, let $q_1$ be the smallest $q_i$. Then $b_{q_1}\leq 4d$.
\end{prop}
\begin{proof}
Let $n=2^{e}m$ where $m$ is odd. First, we derive that $e\leq 2$. Suppose that $e\geq 1$. We can deduce by \ref{fact4} that $v_{2^{e}}\equiv 3\pmod 4$. Hence, $v_{2^{e}}$ has a prime factor $q\equiv 3\pmod 4$. Reducing \ref{fact3} modulo $q$, we obtain $(-b^2-4\mid q)=1$. Since $q\equiv 3\pmod 4$, we have $(b^2+4\mid q)=-1$. By Lemma \ref{lem:4}, we thus have $z(q)\mid q+1$. Since $q\mid v_{2^{e}}$, we have $q\mid u_{2^{e+1}}$ by \ref{fact2}. Using \ref{fact3}, we get $q\nmid u_{2^{e}}$. Thus, $z(q)=2^{e+1}$. By Lemma \ref{lem:4}, we thus have $2^{e+1}\mid q+1$. By \ref{fact5}, $\varphi(v_{2^{e}})\mid\varphi(v_n)$. So $q-1\mid\varphi(v_n)$ and so $q=2^{e_1}3^{e_2}+1$ for some integers $e_1,e_2\geq 0$. Since $q\equiv 3\pmod 4$, $e_1=1$. Thus, $2^{e+1}\mid 2\cdot 3^{e_2}+2$ so that $2^{e}\mid 3^{e_2}+1$. Depending on the partity of $e_2$, we have $\nu_2(3^{e_2}+1)=1,2$ and so $e\leq 2$.

Now let $n=2^{e}3^{\beta}m$ where $m$ is not divisible by $2$ or $3$. Assume that $e\geq 1$ and $\beta\geq 2$. Notice that $v_3=b^3+3b$, which is even. Thus, by \ref{fact3} and \ref{fact4}, we have that $v_{2^{e}3^{\beta-1}}$ is also even. Thus, $v_{2^{e}3^{\beta-1}}^2-1\equiv 3\pmod 4$ and so there exists a prime factor $q$ of $v_{2^{e}3^{\beta-1}}^2-1$ such that $q\equiv 3\pmod 4$. By \ref{fact7}, $q\mid v_{2^e3^{\beta}}$. By \ref{fact2} and \ref{fact3}, we have $q\mid u_{2^{e+1}3^{\beta}}$ and $q\nmid u_{2^{e}3^{\beta}}$. Also, by \ref{fact3}, we have $(b^2+4)u_{2^{e+1}3^{\beta-1}}\equiv v_{2^{e+1}3^{\beta-1}}^2-4\equiv -4\pmod q$ and so $q\nmid u_{2^{e+1}3^{\beta-1}}$. Thus, $z(q)=2^{e+1}3^{\beta}$. By \ref{fact2}, we have $(b^2+4)u_{2^{e}3^{\beta-1}}^2+4=v_{2^{e}3^{\beta-1}}^2$. Modulo $q$, we can derive that $(-b^2-4\mid q)=1$. Since $q\equiv 3\pmod 4$, we have $(b^2+4\mid q)=-1$. By Lemma \ref{lem:4}, we thus have $z(q)\mid q+1$ and so $2^{e+1}3^{\beta}\mid q+1$. By \ref{fact5}, we have $\varphi(v_{2^{e}3^{\beta}})\mid\varphi(v_n)$. So $q-1\mid\varphi(v_n)$ and so $q=2^{e_1}3^{e_2}+1$ for some nonnegative integers $e_1$ and $e_2$. Since $q\equiv 3\pmod 4$, $e_1=1$.  Thus, $2^{e+1}3^{\beta}\mid 2\cdot 3^{e_2}+2$ so that $2^{e}3^{\beta}\mid 3^{e_2}+1$. But $\beta\geq 2$ and so $3\mid 3^{e_2}+1$, which cannot happen. Thus, either $e=0$ or $\beta\leq 1$. Thus, if there is no prime greater than $3$ dividing $n$, then $n=0,1,2,3,4,6,12$ or $n$ is a power of $3$.

Assume that $p>3$ is a prime factor of $m$. By \ref{fact5}, $v_{2^{e}p}$ has the same property that its Euler function is divisible only by primes which are at most $3$. By Carmichael's Theorem, there exists a prime $q$ dividing $u_{2^{e+1}p}$ such that $q$ does not divide $u_n$ for all $n<2^{e+1}p$ so that $z(q)=2^{e+1}p$. By \ref{fact2}, we have $u_{2^{e+1}p}=u_{2^{e}p}v_{2^{e}p}$ and so $q\mid v_{2^{e}p}$. Notice that since $q\nmid u_{2^{e+1}}$, we have $q\nmid v_{2^{e}}$ by \ref{fact2}. Also, notice that $2\mid b^2+1=u_3$ and so $q$ is odd. If $e=0$, then by \ref{fact3}, we have $(b^2+4)u_p^2-4=v_p^2$, which when reduced modulo $q$, gives us $(b^2+4|q) =1$. If $q=b^2+4$, then $q\mid 4$ and $q>4$, a contradiction. Thus, $q\neq b^2+4$. By Lemma \ref{lem:4}, we have $q\equiv 1\pmod {2p}$, therefore $p\mid\varphi(v_{p})$, which is a contradiction because $p>3$. This shows that the only potential solutions when $e=0$ occur when $n$ is a power of $3$. Assume now that $e\geq 1$. By Lemma \ref{lem:4}, we have $(b^2+4)u_{2^{e}p}^2+4=v_{2^{e}p}^2$, which when reduced modulo $q$ gives us $(-b^2-4|q)=1$. If $q\equiv 1\pmod 4$, then we obtain  $(b^2+4|q)=1$ from which again we can deduce that $q\equiv 1\pmod p$ by Lemma \ref{lem:4}. Hence, $p\mid\varphi({v_{2^{e}p}})$, which is a contradiction for $p>3$. Thus, we may assume that $q\equiv 3\pmod 4$ for all prime factors $q$ of $v_{2^{e}p}/v_{2^{e}}$. Thus, for each such $q$, we have $q = 2\cdot 3^{b_q}+1$ and $(b^2+4|q)=-1$. By Lemma \ref{lem:4}, we have $q\equiv -1\pmod p$ so that $2\cdot 3^{b_q}+1={a_q}p-1$ for some even integer $a_q$. Suppose that $3\mid v_{2^{e}p}/v_{2^{e}}$. Then $3\mid v_{2^{e}p}$. By \ref{fact3}, we have $3\nmid u_{2^{e}p}$. From $u_2=b$ and $u_4=b^3+2b=b(b^2+2)$, we can derive that $z(3)=2$ or $4$. Hence, $e=1$ and $z(3)=4$. Thus, $3\nmid b=v_2$, $\nmid u_2$, and $3\mid u_4$. But by \ref{fact2}, we have $u_4=u_2v_2$, a contradiction. Hence, $3\nmid v_{2^{e}p}/v_{2^{e}}$. Since $v_{2^{e}p}/v_{2^{e}}$ is odd and not divisible by $3$, we can deduce that $v_{2^{e}p}/v_p$ is squarefree since its Euler function is divisible only by primes which are at most $3$. Thus, we get that
\begin{equation*}
v_{2^{e}p}=v_{2^{e}}q_1q_2\cdots q_l
\end{equation*}
where $q_i=2\cdot 3^{b_{q_i}}+1$ for $i=1,\ldots, l$. We may assume that $1\leq b_{q_1}<\ldots<b_{q_l}$. By \ref{fact6}, we have that
\begin{equation*}
3^{b_1}\mid v_{2^{e}p}-v_{2^{e}}=(b^2+4)u_{2^{e-1}(p-1)}u_{2^{e-1}(p+1)}.
\end{equation*}
We know that at least one of $u_{2^{e-1}(p-1)}$ and $u_{2^{e-1}(p+1)}$ is divisible by $3$. Pick the value $d$ such that \ref{fact8} holds. Then by \ref{fact8}, we have
\begin{align*}
&\min\{\nu_3(u_{2^{e-1}(p-1)},\nu_3(u_{2^{e-1}(p+1)})\}\leq d\\
&\max\{\nu_3(u_{2^{e-1}(p-1)},\nu_3(u_{2^{e-1}(p+1)})\}\leq d+\max\{\nu_3(p-1),\nu_3(p+1)\}.
\end{align*}
Also, we have $b_{q_1}\leq\nu_3(u_{2^{e-1}(p-1)})+\nu_3(u_{2^{e-1}(p+1)})$ and so the first inequality implies
\begin{equation*}
\max\{\nu_3(u_{2^{e-1}(p-1)},\nu_3(u_{2^{e-1}(p+1)})\}\geq b_{q_1}-d.
\end{equation*}
Thus, we have $b_{q_1}-2d\leq\max\{\nu_3(p-1),\nu_3(p+1)\}$. Assume that $b_{q_1}\geq 2d$. Then either $3^{b_{q_1}-2d}\mid (p-1)/2$ or $3^{b_{q_1}-2d}\mid (p+1)/2$. Since $p=\frac{2\cdot 3^{b_{q_1}}+2}{2}$, we can therefore derive that $3^{b_{q_1}-2d}\mid a_{q_1}+2$ or $3^{b_{q_1}-2d}\mid a_{q_1}-2$. Since $(p+1)/2\geq 3^{b_{q_1}-2d}$, we obtain
\begin{equation*}
\frac{3^{b_{q_1}}+1}{a_{q_1}}=\frac{p}{2}>3^{b_{q_1}-2d}-1.
\end{equation*}
If $a_{q_1}\geq 3^{2d}+1$, then we have
\begin{equation*}
3^{b_{q_1}}+1>(3^{2d}+1)(3^{b_{q_1}-2d}-1)=(3^{2d}+1)3^{b_{q_1}-2d}-3^{b_{q_1}}-1=3^{b_{q_1}-2d}-1,
\end{equation*}
which implies that $b_{q_1}\leq 4d$. On the other hand, if $a_{q_1}\leq 3^{2d}-1$, we have
\begin{equation*}
3^{b_{q_1}}-2d\leq a_{q_1}+2\leq 3^{2d}+1,
\end{equation*}
which again implies that $b_{q_1}\leq 4d$. Thus, we have our result.
\end{proof}
\begin{proof}[Proof of Theorem \ref{thm:phiresult5}]
Let $c=1$. Let $b=3$. We have $3^2+4=13$ is prime. We can check that for $n\leq 12$ the only solutions are as stated. Also, we can verify that $17\mid\varphi(v_{27})$ and so $n=9$ is the highest power of $3$ that gives a solution. By Proposition \ref{prop:4}, we may assume that $v_{2^{e}p}$ has a prime factor $q$ among $7,19$, and $163$, but that this prime factor does not divide $v_{2^{e}}$ where $2^{e}p\mid n$ with $e=1$ or $2$ and $p>3$ is prime. Suppose $q=7$. We can verify that $z(7)=8$ and so $4\mid n$ or $e=2$. But then $7\mid v_{2^{e}}$, a contradiction. Now suppose that $q=19$. We can deduce that $(13|19)=-1$ and so we have $p\mid 20$ and so $p=5$. We can verify that $19\mid v_{10}$, but $19\nmid v_{20}$ so that $e=1$. But $5\mid\varphi(v_{10})$, a contradiction. Finally, assume that $q=163$. We can deduce that $(13|163)=-1$ and so we have $p\mid 164$ and so $p=41$. We can verify that $e=1$. But $41\mid\varphi(v_{82})$ and so again we get a contradiction. Thus, all of the solutions are as stated.

Let $b=5$. We have $5^2+4=29$ is prime. We can check that for $n\leq 12$ the only solutions are as stated. Also, we can verify that $11\mid\varphi(v_{9})$ and so $n=3$ is the highest power of $3$ that gives a solution. By Proposition \ref{prop:4}, we may assume that $v_{2^{e}p}$ has a prime factor $q$ among $7,19,163,487,1459$, and $39367$, but that this prime factor does not divide $v_{2^{e}}$ where $2^{e}p\mid n$ with $e=1$ or $2$ and $p>3$ is prime. Suppose $q=7$. We can verify that $z(7)=6$. But $7\mid v_{2^{e}p}$ implies $7\mid u_{2^{e+1}p}$, which cannot happen because $6\nmid 2^{e+1}p$. Now suppose that $q=19$. We can deduce that $(29|19)=-1$ and so we have $p\mid 20$ and so $p=5$. We can verify that $19\mid v_{10}$, but $19\nmid v_{20}$ so that $e=1$. But $17\mid\varphi(v_{10})$, a contradiction. Next, assume that $q=163$. We can deduce that $(29|163)=-1$ and so we have $p\mid 164$ and so $p=41$. Suppose that $e=2$. Then $163\mid v_{164}=v_{82}^2-2$, which implies that $(2|163)=1$, which is false. Thus, $e=1$. But $5\mid\varphi(v_{82})$, a contradiction. If $q=487$ or $1459$ we can deduce that $(29|q)=1$ and so we have $p\mid q-1$, which is not possible since $p>3$. If $q=39367$, then we can deduce that $(29|39367)=-1$ and so we have $p\mid 39368$. Thus, $p=7,19$, or $37$. We therefore have six choices for $2^{e}p$: $14,28,38,76,74,148$. But checking each of these, we deduce that $39367\nmid v_{2^{e}p}$, a contradiction. Thus, all of the solutions are as stated.

Let $b=7$. We have $7^2+4=53$ is prime. We can check that for $n\leq 12$ the only solutions are as stated. Also, we can verify that $17\mid\varphi(v_{9})$ and so $n=3$ is the highest power of $3$ that gives a solution. By Proposition \ref{prop:4}, we may assume that $v_{2^{e}p}$ has a prime factor $q$ among $7,19$, and $163$, but that this prime factor does not divide $v_{2^{e}}$ where $2^{e}p\mid n$ with $e=1$ or $2$ and $p>3$ is prime. Suppose $q=7$. By a congruence argument, we can deduce that $n$ must be odd, contradicting $e=1$ or $2$. Now suppose that $q=19$. We can deduce that $(53|19)=-1$ and so we have $p\mid 20$ and so $p=5$. We can verify that $19\mid v_{10}$, but $19\nmid v_{20}$ so that $e=1$. But $137\mid\varphi(v_{10})$, a contradiction. Finally, assume that $q=163$. We can deduce that $(53|163)=1$ and so we have $p\mid 162$, which is not possible since $p>3$ and so again we get a contradiction. Thus, all of the solutions are as stated.
\end{proof}
\section{Acknowledgements}
The author would like to thank Dr. Daniel Berend and Dr. Florian Luca for their suggestions with this paper and the Azrieli Foundation for the award of an Azrieli International Postdoctoral Fellowship, which made this research possible.
\bibliographystyle{plain}
\bibliography{diophantinepaper}
\end{document}